\documentclass[10pt,regno]{amsart}
\usepackage[normalem]{ulem}

\usepackage{euscript,graphicx,overpic,bm}
\usepackage{amsmath,amsthm,amsfonts,amssymb,amscd}
\usepackage{mathrsfs,color}

\newtheorem{theo}{Theorem}[section]
\newtheorem{theorem}{Theorem}
\newtheorem*{theorem*}{Theorem}

\newtheorem{lemma}[theo]{Lemma}

\newtheorem{corollary}[theo]{Corollary}
\newtheorem{proposition}[theo]{Proposition}
\newtheorem{definition}[theo]{Definition}
\theoremstyle{definition}
\newtheorem{example}{Example}
\newtheorem{remark}[theo]{Remark}
\newtheorem*{remark*}{Remark}

\newcommand{\eqdef}{\stackrel{\scriptscriptstyle\rm def}{=}}

\DeclareMathOperator{\card}{card}

\DeclareMathOperator{\Per}{Per}

\DeclareMathOperator{\supp}{supp}

\DeclareMathOperator{\ch}{ch}

\def\blambda{\bm{\lambda}}

\def\bZ{\mathbb{Z}}
\def\bR{\mathbb{R}}

\def\cRR{\EuScript{R}}

\def\cone{\mathscr{C}}

\def\cO{\EuScript{O}}

\def\cU{\EuScript{U}}

\def\cM{\EuScript{M}}

\def\La{\Lambda}

 \def\NN{{\mathbb N}}  \def\PP{{\mathbb P}}
 \def\RR{{\mathbb R}}  \def\TT{{\mathbb T}}
   
 \def\ZZ{{\mathbb Z}}

\DeclareMathSymbol{\varnothing}{\mathord}{AMSb}{"3F}

\title[Dominated Pesin theory]{Dominated Pesin theory:\\ convex sum of hyperbolic measures}

\author{Jairo Bochi}\address{Facultad de Matem\'aticas, Pontificia Universidad Cat\'olica de Chile. Avenida Vicu\~na Mackenna 4860, Santiago, Chile}
\email{jairo.bochi@mat.uc.cl}
\urladdr{http://www.mat.uc.cl/~jairo.bochi/}

\author{Christian Bonatti} \address{CNRS and IMB, Universit\'e de Bourgogne, Dijon, France}
\email{bonatti@u-bourgogne.fr}
\urladdr{http://bonatti.perso.math.cnrs.fr/}

\author{Katrin Gelfert} \address{IM, Universidade Federal do Rio de Janeiro, Cidade
Universit\'aria - Ilha do Fund\~ao, Rio de Janeiro 21945-909, Brazil}
\email{gelfert@im.ufrj.br}
\urladdr{http://www.im.ufrj.br/~gelfert/}



\thanks{This paper was supported by Ci\^encia Sem Fronteiras CAPES and CNPq (Brazil). The authors acknowledge the kind hospitality of 
Institut de Math\'ematiques de Bourgogne, Dijon (France) and of Departamento de Matem\'atica, PUC-Rio (Brazil).}

\begin{document}
\maketitle
\begin{abstract}
In the uniformly hyperbolic setting it is well known that the set of all
measures supported on periodic orbits is dense in the convex space of all invariant measures.  
In this paper we consider the converse question, in the non-uniformly hyperbolic setting: assuming that some
ergodic measure converges to a convex combination of hyperbolic ergodic measures, what can we deduce about the initial measures?

To every hyperbolic measure $\mu$ whose stable/unstable Oseledets splitting is dominated we associate canonically a  unique class $H(\mu)$ of periodic orbits for the homoclinic relation, called its \emph{intersection class}. In a dominated setting, we prove that a measure for which almost every measure in its ergodic decomposition is hyperbolic with the same index such as the dominated splitting is accumulated by ergodic measures if, and only if, almost all such ergodic measures have a common intersection class.

We provide examples which indicate the importance of the domination assumption.\end{abstract}

\section{Introduction} 

\subsection{Quick presentation of the results}

The space $\cM_f$ of invariant measures by a homeomorphism $f$ of a compact metric space is a compact metric space 
(for the the weak$\ast$ topology) and is convex. We denote by $D$ a distance on $\cM_f$ defining the weak$\ast$ topology. The ergodic measures are the extremal points of this convex set and 
any invariant measure can be written  as a unique convex sum of ergodic measures  and 
called its  \emph{disintegration in ergodic measures} or \emph{$f$-ergodic decomposition}.

A typical picture of hyperbolic dynamics is that the ergodic measures associated to periodic orbits may be dense in $\cM_f$. 
More precisely, if you consider a shift or a subshift
of finite type, there are periodic orbits following an arbitrary itinerary. Hence, given any  ergodic measures $\mu_1,\ldots,\mu_k$, 
there are periodic orbits
which follow a given proportion of time a typical point of the measure $\mu_1$ in an orbit segment long enough for approaching $\mu_1$ 
and then follow 
$\mu_2$ and so on, so that the measure obtained at the period is arbitrarily close to a given convex combination of the $\mu_i$. 
Now, if $f$ is a diffeomorphism on a manifold and if $\Lambda$ is an invariant (uniformly) hyperbolic basic set, 
the existence of a Markov partitions allows us to transfer this property to the set $\cM_f(\Lambda)$ of invariant measures supported on $\Lambda$.

To understand the non-hyperbolic dynamical systems, one often starts the study by considering the hyperbolic parts contained in $f$. For instance, 
 if $p$ and $q$ are hyperbolic periodic points with the same index, one says that they  are \emph{homoclinically related}  if 
 the stable and unstable manifold of $p$ cut transversely the unstable and stable  manifolds of $q$, respectively.
 A famous result by Smale shows that, if $p$ and $q$ are homoclinically related, then they are contained in a hyperbolic basic set.  
 Thus, any convex combination of the invariant measures associated to $p$ and $q$ is accumulated by periodic orbits.

 In this paper, we are interested in some kind of converse to this property. 
 Given a diffeomorphism $f$ of a compact manifold $M$ and given two
 ergodic measures $\mu_1$ and $\mu_2$:
 \emph{what can we say about the dynamical behavior of $f$ if we know that a convex combination 
 $\alpha\mu_1+(1-\alpha)\mu_2$ is the limit of measures associated to periodic orbits?}
In Section~\ref{sec:examples} we give examples showing that this approximation of a convex combination does not 
imply anything about homoclinic relations, even in the $\bR$-analytic setting, if the dynamic fails to be dominated.
 
 Our main result states that, if the stable/unstable splitting is  dominated, 
 then  \emph{the approximation of convex combinations  by ergodic measures is equivalent
 to transverse homoclinic intersection}. The idea consists  in controlling the size of the 
 Pesin's invariant manifolds at the time the orbits are following generic points of one or the other measure, 
 and to check that this size is enough for getting transverse intersections. Usually, Pesin theory does not hold
 in the $C^1$ setting (see \cite{Pug:84,BonCroShi:13}). However, \cite{AbdBonCro:11} shows 
 that Pesin theory holds for $C^1$-diffeomorphism if one assume that the 
 stable/unstable splitting is  dominated, which is precisely our setting. This explains why our statements 
 hold for $C^1$-diffeomorphisms.  Let us emphasize that our results are not perturbative results: 
 we are not creating periodic orbits, measures, or homoclinic intersections 
 by using some perturbation lemma, but prove that they already exist. 
 
Let us now present precisely this result. 

\medskip\noindent{\textbf{Standing hypothesis.} 
Let $f$ be a $C^1$ diffeomorphism of a Riemannian manifold $M$, $\Lambda\subset M$ a invariant compact  set which admits a dominated splitting $E\oplus_{_<}F$ with $s$-index $\dim E$. Let $\tilde\Lambda$ be the maximal invariant set in a neighborhood of $\Lambda$ so that the dominated splitting $E\oplus_{_<}F$ extends to $\tilde\Lambda$.

\medskip 
Here recall that $T_\Lambda M=E\oplus_{_<}F$ \emph{dominated splitting}  over a set $\Lambda$ if the bundles $E$ and $F$ are $df$-invariant and there is a Riemannian metric over $M$ so that $df$ expands the vectors in $E_x$ strictly less than in $F_x$. We call $\dim E$ the \emph{$s$-index} of the splitting (for further details see Section~\ref{sec:perhypdom}). 
Let $\mu$ be an $f$-invariant Borel probability measure supported on $\Lambda$. We say that $\mu$ is \emph{hyperbolic}  if its Lyapunov exponents are nonzero almost everywhere, and if almost all points have the same number of negative Lyapunov exponents, in this case we call this number the \emph{$s$-index} of $\mu$. The \emph{stable/unstable splitting} of $\mu$ is 
 the Oseledets splitting $E^s\oplus_{_<} E^u$ defined over  the generic points of $\mu$, where $E^s$ ($E^u$) is the sum of the Lyapunov spaces corresponding to the negative (positive) Lyapunov exponents. If $\mu$ supported on $\Lambda$ is hyperbolic with
 $s$-index of $\mu$ equal to the $s$-index of 
the dominated splitting $E\oplus_{_<}F$, then $E^s_x=E_x$ and $E^u_x=F_x$ for any $\mu$-generic point $x$. One then says that  \emph{the stable/unstable splitting of $\mu$ is dominated}.
 
 If $\mu$ is a hyperbolic ergodic  measure whose stable/unstable splitting is dominated, then \cite[Proposition 1.4]{Cro:11} 
 (see Corollary~\ref{corpro:sylvnew})
 shows that every generic point $x$ of $\mu$ is the limit of  periodic orbits approaching $\mu$ in 
 the weak$\ast$ topology, and approaching the support of $\mu$ in the Hausdorff topology; these periodic orbits
 are hyperbolic of the same $s$-index as $\mu$; by the subadditive and maximal ergodic theorems we also verify that there is a uniform proportion of times on these orbits for which the size of their invariant manifolds is uniform (see Proposition~\ref{pro:ABapprox}). As a consequence we prove the following.
 
 \begin{theorem}\label{t.homoclinic}
Let $f$ be a $C^1$ diffeomorphism of a Riemannian manifold $M$, $\Lambda\subset M$ a compact invariant set with a dominated splitting $E\oplus_{_<}F$ with $s$-index $\dim E$, 
and $\mu$ a hyperbolic  invariant (not necessarily ergodic) Borel probability measure supported in  $\Lambda$ with $s$-index $\dim E$.
Let $\tilde\Lambda$ be the maximal invariant set in a neighborhood of $\Lambda$ so that 
the dominated splitting $E\oplus_{_<}F$ 
extends to $\tilde\Lambda$.
	
There exists a positive number $\varkappa=\varkappa(\mu)$ such that any pair $p_1,p_2\in\tilde\La$ of periodic points so that the corresponding invariant measures  $\nu_1, \nu_2 \in\cM_f(\tilde\Lambda)$ 
satisfy $D(\nu_i,\mu)<\varkappa$, $i=1,2$,  are homoclinically related.
\end{theorem}
 
 We consider the homoclinic relation on the set of hyperbolic periodic orbits in $\tilde\Lambda$. It is an 
 equivalence relation and  we call \emph{intersection classes} (relative to $\tilde\Lambda$) the equivalence classes for the homoclinic relation which are contained in $\tilde\Lambda$.%
\footnote{The \emph{intersection class} of a hyperbolic periodic point was first considered in \cite{New:80} and called \emph{$h$-class}. With this terminology, the classical \emph{homoclinic class} of a hyperbolic periodic point (called \emph{$h$-closure} in \cite{New:80}) is the closure of its intersection class. See also Remark~\ref{rem:hertzens} for similar, but \emph{a priori} unrelated, concepts.} 
Under the hypotheses of Theorem~\ref{t.homoclinic}, \cite[Proposition 1.4]{Cro:11} together with Theorem~\ref{t.homoclinic}  associates  to any hyperbolic invariant measure $\mu$ which is accumulated by ergodic measures (and, in particular, to any hyperbolic ergodic measure $\mu$) a unique intersection class, denoted by $H(\mu)$ and called the \emph{intersection class of $\mu$} (relative to $\tilde\Lambda$): the class $H(\mu)$ is  the intersection class of any periodic orbit in $\tilde \La$ whose corresponding measure is close enough to $\mu$. Since we always will restrict our considerations to $\tilde\Lambda$ alone, below we simply will talk about \emph{intersection classes}.
We will prove Theorem~\ref{t.homoclinic} at the end of Section~\ref{sec:intcla}.

The following result is an immediate consequence of Theorem~\ref{t.homoclinic}.
 
 \begin{corollary}\label{c.homoclinic} 
 	Under the hypotheses of Theorem~\ref{t.homoclinic},  any ergodic measure 
 $\nu\in \cM_f(\tilde\Lambda)$ with $D(\nu,\mu)<\varkappa(\mu)$  shares the same intersection class, that is, we have
\[
	H(\nu)=H(\mu).
\] 
\end{corollary}

 As noticed above, convex combinations of the measures associated to homoclinically related periodic orbits are limits of periodic orbits 
 in the same class. Observing that if $\mu$ is hyperbolic ergodic  then it can be approximated by periodic orbits (see Corollary~\ref{corpro:sylvnew}), as a direct consequence we get the following.
 
\begin{remark}\label{r.homoclinic} 
Under the hypotheses of Theorem~\ref{t.homoclinic}, if $\mu,\nu\in \cM_f(\La)$ are hyperbolic ergodic with $s$-index $\dim E$ such that $\nu$ has the same intersection class as $\mu$, that is, $H(\nu)=H(\mu)$, then any convex combination $\alpha\mu+(1-\alpha)\nu$, $\alpha\in[0,1]$, is the weak$\ast$ limit of measures associated to periodic orbits in $\tilde\Lambda$ belonging to this class.  
\end{remark}

The main purpose of this paper is to prove the converse to Remark~\ref{r.homoclinic}.  
 We are now ready to state our main result in that direction.

 \begin{theorem}\label{the:sinterclass} 
Let $f$ be a $C^1$ diffeomorphism of a Riemannian manifold $M$ and $\Lambda\subset M$ be a compact invariant set with a dominated splitting $T_\Lambda M=E\oplus_{_<}F$, and $\mu\in\cM_f(\Lambda)$ be a hyperbolic measure with $s$-index equal to $\dim E$. Let $\tilde \Lambda$ be the maximal invariant set in a neighborhood of $\Lambda$ so that the dominated splitting $E\oplus_{_<}F$ extends to $\tilde \Lambda$. Let $\mu=\int\nu\,d\blambda(\nu)$ be the ergodic decomposition of $\mu$. Then $\mu$ is accumulated by ergodic measures $\nu_n\in\cM_f(\tilde\Lambda)$ if, and only if, $\blambda$-almost all ergodic measures $\nu$ are hyperbolic with index $\dim E$ and have the same intersection class $H=H(\nu)$. Moreover, in this case $H=H(\mu)$.
\end{theorem}

To sketch the proof of Theorem~\ref{the:sinterclass}, assume that $\mu=\int\nu\,d\blambda(\nu)$ is approached by ergodic measures $\nu_n$ in $\cM_f(\tilde \La)$. Then applying Theorem~\ref{t.homoclinic} we first verify that the measures $\nu_n$ are hyperbolic with the same $s$-index as $\blambda$-almost every $\nu$, hence as the dominated splitting $E\oplus_{_<}F$ and that each measure $\nu_n$ is a limit of measures supported on hyperbolic periodic orbits with index $\dim E$ in $\tilde \La$. Thus, one can assume that the measures $\nu_n$ are supported on periodic orbits $\cO_n$. The key step in the proof of Theorem~\ref{the:sinterclass} is to prove that $\cO_n$ have invariant manifolds of uniform size at points of the orbits close to generic points of a $\blambda$-typical ergodic measure $\nu$. This proof will be completed in Section~\ref{sec:prooof}.
 
In Section~\ref{sec:examples} we will give examples of smooth diffeomorphisms having hyperbolic saddle fixed points $p_1,p_2$ such that a convex combination of their Dirac measures is approached by hyperbolic periodic orbits, even with exponents far away from $0$,  and  whose homoclinic and intersection classes  are disjoint (see Theorem~\ref{the:exeunihyp}). This illustrates the importance of the domination assumption in our results.

\subsection{Motivation and historical setting}

The question about ergodic measures associated to periodic orbits being dense in the space 
$\cM_f$ of $f$-invariant measures has been studied also in a more abstract setting, 
that is, without \emph{a priori} assuming that there is a differentiable hyperbolic dynamics present. 
Perhaps among the first attacking this general question was Sigmund~\cite{Sig:70,Sig:74} 
who showed that provided the dynamical system satisfies a so-called \emph{periodic orbit 
specification property} the ergodic measures and, in particular, the periodic orbit 
measures are dense in $\cM_f$. Sigmund's theorem applies to basic sets of axiom A 
diffeomorphisms~\cite{Sig:70,Sig:74}. Roughly speaking, the specification property says that given an arbitrary 
number of arbitrarily long orbit segments, one can find a periodic orbit which stays $\varepsilon$-close 
to each of those segments and between one segment and the next one needs a fixed number of iterations 
which only depends on $\varepsilon$ (see~\cite{Sig:74} for details). 
In the context of a basic set of an axiom A diffeomorphism, the existence of a Markov partition  
guarantees a symbolic 1-1 description of essentially all orbits. The shadowing property enables 
to arbitrarily connect given (``specified'') orbit pieces of arbitrary length  given by certain 
symbolic sequences. Concatenating these symbolic sequences infinitely often, leads to periodic 
shadowing orbits. As the Markov partition can be chosen with arbitrarily small diameter, this 
procedure enables the ``specification'' and production of periodic orbits with an arbitrary precision. 
Sigmund's specification property has been verified also in a number of topological dynamical systems 
such as, for example, topologically mixing subshifts of finite type~\cite{DenGriSig:76} and continuous 
topologically mixing maps on the interval~\cite{Blo:83,Buz:97}. 

The shadowing lemma holds more generally for transitive (uniformly) hyperbolic sets with some caution: 
if $\Lambda$ is a topologically transitive hyperbolic set, it admits a compact neighborhood in which the maximal invariant 
set 
$\tilde \Lambda$ is a topologically transitive hyperbolic set, and pseudo orbits in $\Lambda$ are shadowed by orbits in 
$\tilde\Lambda$. As a consequence, the convex hull of all invariant measure supported in $\Lambda$ is contained
in the closure of measures supported on periodic orbits contained in $\tilde \Lambda$. 

For a general diffeomorphism, the periodic orbit specification property 
beyond a uniformly hyperbolic context is often difficult to verify or fails to hold true. 
There exist various extensions of Sigmund's result under weaker assumptions. 
Perhaps the most interesting line of extension in the spirit of our paper are the 
so-called \emph{approximate product property}~\cite{PfiSul:05} and 
\emph{$g$-almost product property}~\cite{PfiSul:07} initiated by Pfister and Sullivan. 
Their variations of the specification property, roughly speaking, ``allow to make some 
number of errors'' in the resulting shadowing orbits up to some number of iterations which decays 
sufficiently fast when the length of the specified orbits grows to infinity. Under any of their 
conditions they show that in the space of invariant measures ergodic 
measures are still dense.

Much closer to our approach in the present paper is perhaps Hirayama~\cite{Hir:03} 
who studied a general $C^{1+\alpha}$ topologically mixing diffeomorphism preserving an ergodic hyperbolic 
probability measure. He shows that there exists a measurable set $\Gamma$ of  
full measure $\mu$ such that the set of all  measures supported on periodic hyperbolic orbits is 
dense in the (convex) space of invariant measures supported by $\Gamma$, meaning that $\Gamma$ has full measure for any measure in this space. 

All these studies tried to established denseness of ergodic (periodic) measures. 
However, to the best of our knowledge, little is known about dynamical systems which 
do not have such a property. Clearly, one has to disregard cases where the system dynamically 
splits into ``basic pieces'' such as, for example, attractor-repeller pairs or a similar family 
of  unrelated disjoint compact invariant sets. More precisely, Conley theory~\cite{Co:78} divides the dynamics of any
homeomorphism on a compact metric space into \emph{chain recurrent classes}. A non-ergodic measure $\mu$ cannot be
approached by ergodic measures if it is not supported on a unique chain recurrent class, 
and in particular if its disintegration into ergodic measures gives positive weight to measures supported in 
distinct chain recurrent classes. 

However, the undecomposability property of the chain recurrent classes is  a very weak property and
it is easy to 
build examples where non-ergodic measures  are not approached by ergodic ones\footnote{For example, consider a diffeomorphism
of the circle with fixed points which all are of saddle-node type.}. 
Another natural candidate for being an ``elementary piece'' of a diffeomorphisms is the 
\emph{homoclinic class} of a hyperbolic periodic orbit, which is the closure of their 
homoclinic intersection. For $C^1$-generic diffeomorphisms, the homoclinic classes are the chain recurrence classes
containing periodic points (see~\cite{BonCro:04}), leading to the impression that in this setting the basic pieces of the dynamics
are well defined. The homoclinic classes 
are the closure of an increasing sequence of hyperbolic basic sets, 
leading to a good understanding of at least a part of the dynamics contained in it. 
However, the ergodic theory inside 
homoclinic classes is in general not understood. \cite[Conjecture 2]{Bon:11} proposes that  
for $C^1$-generic diffeomorphisms the ergodic measures supported in a homoclinic class are approached by periodic orbits
contained in the class. 
In the opposite direction, in \cite{DiaGel:12, DiaGelRam:14} there can be found an example of a diffeomorphism having a nontrivial homoclinic class with a hyperbolic 
periodic orbit which is isolated from the other ergodic measure supported on the class. 
We note that this example fits into the hypothesis of Theorem~\ref{the:sinterclass}. 
But,  besides the existence of examples, it remains
\emph{a priori} unclear what general mechanism causes that the closure of the set of invariant ergodic probability measures splits 
into distinct components. 
According to our main results, 
we can now give a refined statement towards the above conjecture:   
under the assumptions of $C^1$ domination, for an individual diffeomorphism 
any ergodic hyperbolic measure can be approached only by periodic orbits contained in its  intersection class. 

\smallskip\noindent
\textbf{Acknowledgement:} We thank Lorenzo J. D\'\i az for discussion at the origin of this paper and of the examples in Section~\ref{sec:examples}. We also thank the referee whose suggestions lead us to a more conceptional and much easier definition of the intersection class of a hyperbolic ergodic measure.

\section{Preliminaries} 

\subsection{The space of measures}\label{ssec:choquet}

Let $X$ be a compact metric space and denote by $\cM(X)$ the set of Borel probability measures on $X$. It is well known that it is a compact metrizable topological space when  equipped with the weak$\ast$ topology (\cite[Chapter 6.1]{Wal:82}).  If $\{\varphi_n\}_{n\ge1}$ is a dense subset of the space $C^0(X)$ of continuous functions on $X$ then
\[
	D(\mu,\nu)
	\eqdef\sum_{n=1}^\infty \frac{1}{2^n\lVert \varphi_n\rVert_\infty} \,
		\Big\lvert \int\varphi_n\,d\mu - \int\varphi_n\,d\nu\Big\rvert,
	\quad \lVert \varphi_n\rVert_\infty
	:= \sup_{x\in\Lambda}\,\lvert\varphi_n(x)\rvert	
\]
provides a metric on this space giving the weak$\ast$ topology.  We will use the fact that in the weak$\ast$ topology $\mu_n\to\mu$ if, and only if, for every $\varphi\in C^0(X)$ we have $\int\varphi\,d\mu_n\to\int\varphi\,d\mu$. In particular, we can consider $\cM(\cM(X))$. If $\blambda\in\cM(\cM(X))$, then $\mu=\int\nu\,d\blambda(\nu)$ is a well defined element in $\cM(X)$ and we call $\blambda$ a \emph{decomposition} of $\mu$.

If $f\colon X\to X$ is a continuous map on $X$, denote by $\cM_f(X)\subset\cM(X)$ the subspace of all $f$-invariant Borel probability measures on $X$ and  by $\cM_{\rm erg}(X)\subset\cM_f(X)$ the subset of $f$-ergodic measures. Recall that $\cM_f(X)$ is a non-empty Choquet simplex (see~\cite[Chapter 6.2]{Wal:82}). In particular, it is convex and compact. The extreme points of $\cM_f(X)$ are the ergodic measures. Any point in a Choquet simplex is represented by a unique probability measure on the extreme points -- this point of view is often taken to show the ergodic decomposition of nonergodic measures. Given $\mu\in\cM_f(X)$, let $\blambda\in\cM(\cM(X))$ be the \emph{$f$-ergodic decomposition} of $\mu$, that is, the unique decomposition of $\mu$ such that $\blambda(\cM_{\rm erg}(X))=1$.

We will use the following straightforward result.

\begin{lemma}\label{lem:finiteapprox}
	Let $\mu\in\cM_f(X)$ and let $\blambda\in\cM(\cM(X))$ be its $f$-ergodic decomposition.
Then for any neighborhood $V$ of $\mu$ in $\cM_f(X)$, there exist $\nu_1,\ldots,\nu_k\in\cM_{\rm erg}(X)$ and positive numbers $\lambda_1,\ldots,\lambda_k$ satisfying $\sum_{j=1}^k\lambda_j=1$ such that $\sum_{j=1}^k\lambda_j\nu_j\in V$.
\end{lemma}

\begin{proof}
See, for example,~\cite[Lemma 2.1]{AviBoc:12}. Indeed, the only difference is that we require $\nu_j$ to be ergodic.  Note that in the proof in~\cite{AviBoc:12} one can take $y_i$  ergodic.
\end{proof}

\subsection{Hyperbolicity and dominated splitting}\label{sec:perhypdom}

Now let $f\colon M \to M$ be a $C^1$-diffeomorphism of a Riemannian manifold, and let $ \Lambda \subset X$ be a compact invariant set.
By Oseledets' multiplicative ergodic theorem, given $\mu\in\cM_f(\Lambda)$, for $\mu$-almost every $x$ there are a positive integer $s(x)\leq \dim M$, a  $df$-invariant splitting
\[
T_xM = \bigoplus_{i=1}^{s(x)}E^i_x, 
\]
and numbers $\chi_1(x)<\ldots<\chi_{s(x)}(x)$, called the \emph{Lyapunov exponents} of $x$, such that for all $i=1,\ldots,s(x)$ and $v\in E^i_x\setminus\{0\}$ we have
\[
	\chi_i(x)
	= \lim_{\lvert n\rvert\to\infty}\frac 1n\log\,\lVert df^n_x(v)\rVert.	
\]
We call such a point $x$ \emph{Lyapunov regular} with respect to $f$ (see for example~\cite{Man:87} for details on Lyapunov regularity). Moreover, $\chi_i(\cdot)$ are $\mu$-measurable functions and we denote by
\[
	\chi_i(\mu)\eqdef\int\chi_i(x) d\mu(x).
\]
the Lyapunov exponents of the measure $\mu$ (observe that we allow $\mu$ to be nonergodic).
For a Lyapunov regular point $x$ let us denote by $E^s_x$ (by $E^u_x$) the span of all subspaces of $T_xM$ that correspond to a negative Lyapunov exponent (a positive Lyapunov exponent).
The \emph{stable index} or simply \emph{$s$-index} (\emph{unstable index} or \emph{$u$-index}) 
of a Lyapunov regular point $x$ is the dimension of  $E_x^s$ (of $E_x^u$).

We say that $\mu$ is \emph{hyperbolic} if for $\mu$-almost every $x$ there is $1\le \ell=\ell(x)<s(x)$ such that
\[
	\chi_\ell(x)<0<\chi_{\ell+1}(x)
\]
(there are negative and positive but no zero Lyapunov exponents).
If $\mu$ is ergodic then $s(\cdot)$, $\chi_i(\cdot)$, as well as the dimensions of $E_{(\cdot)}^{s}$ and $E_{(\cdot)}^u$ are constant almost everywhere. Correspondingly to what is defined above, the \emph{stable index} (\emph{unstable index}) of a  hyperbolic measure is the \emph{stable index} (\emph{unstable index}) of almost every Lyapunov regular point.   

An $f$-invariant set $\Gamma\subset M$ is \emph{hyperbolic} if there exists a $df$-invariant  splitting $E^s\oplus E^u=T_\Gamma M$ of the tangent bundle and positive constants $C$ and $\lambda$ such that for every $x\in\Gamma$ for every $n\ge0$ we have
\[\begin{split} 
	\lVert df^n_x(v)\rVert&\le Ce^{n\lambda}\lVert v\rVert 
	\quad\text{for all}\quad v\in E^s_x,\\
	\lVert df^{-n}_x(w)\rVert&\le Ce^{n\lambda}\lVert w\rVert 
	\quad\text{for all}\quad w\in E^u_x	\,.
\end{split}\]
Recall that for a compact $f$-invariant hyperbolic set, up to a smooth change of metric, we can assume  $C=1$. 

A set $\Gamma\subset M$ is \emph{locally maximal} if there exists an open neighborhood 
$U$ of $\Gamma$ such that $\Gamma=\bigcap_{k\in\mathbb Z}f^k(U)$. 
A set $\Gamma\subset M$ is \emph{transitive} if it is the closure of a positive orbit.
Recall that a set  is \emph{basic} (with respect to $f$) if it is compact, invariant, transitive, locally maximal, and hyperbolic. 

Given an $f$-invariant set $\Gamma$, a $df$-invariant splitting $T_\Gamma M=E\oplus F$  is \emph{dominated} if  
there exists $N\ge1$ such that for every point $x\in \Gamma$ and all unit vectors $v\in E_x$ and $w\in F_x$ we have
\[
	\lVert df^N_x(v)\rVert \le \frac 12 \,\lVert df^N_x(w)\rVert
\]
and if $\dim E_x$ (and hence $\dim F_x$) does not depend on $x\in\Gamma$. 
It will be denoted by $E\oplus_{_<}F$. 

Recall that a dominated splitting $E\oplus_{_<} F$ is always continuous and extends to the closure $\overline\Gamma$ of $\Gamma$. Moreover, for a sufficiently small neighborhood $V$ of $\overline\Gamma$, considering the maximal invariant set $\tilde\Gamma$ in $V$, there is a unique dominated splitting on $\tilde\Gamma$ which extends $E\oplus_{_<}F$ (see~\cite{BonDiaVia:05} for more details). 
We call such $\tilde\Gamma$ a \emph{dominated extension} of $(\Gamma,E\oplus_{_<} F)$. 

If $\Gamma\subset \Lambda$ is hyperbolic then its associated splitting $E^s\oplus E^u=T_\Gamma M$ is dominated.

Given a dominated splitting $E\oplus_{_<} F$, it may happen that the bundles $E$ and $F$ can be further decomposed into subbundles satisfying a domination condition. There always exists a (unique) \emph{finest dominated sub-splitting} in the sense that it is not further decomposable. 
For an ergodic measure $\mu\in\cM_f$, the bundles of any dominated  splitting on the support of $\mu$ 
can be written, for $\mu$-generic points, as sums of the Oseledets spaces of groups of increasing Lyapunov exponents. 
We will, however, in this paper disregard such finer splittings. We only consider a dominated
splitting into two bundles which separates positive and negative exponents alone. These subbundles are denoted $E^s$ and $E^u$, respectively, and 
if this splitting is dominated we say that \emph{the stable/unstable splitting of $\mu$ is dominated}. 

We note that, if $E\oplus_{_<}F$ is a dominated splitting defined over $\supp\mu$ which has the same $s$-index such as $\mu$ then the stable/unstable splitting $E^s\oplus E^u$ over $\supp\mu$ is dominated.

\subsection{Subadditive and maximal ergodic theorems}\label{subsec:submaxergthe}

Consider a dominated splitting $E\oplus_{_<} F=T_\Lambda M$. 
Note that the sequence  $(\psi_n)_n$ of functions $\psi_n\colon\Lambda\to\bR$ given by 
\[
	\psi_n(x)
	\eqdef \log\,\lVert df^n_{/E_x}\rVert
\]
is \emph{subadditive}, that is, for every $n,m\ge1$ we have $\psi_{n+m}\le\psi_n+\psi_m\circ f^n$. 
Observe  that the sequence $(-\phi_n)_n$ of functions $-\phi_n\colon\Lambda\to\bR$, where $\phi_n$ are given by
\[
	\phi_n(x)
	\eqdef \log\,\lVert (df^n_{/F_x})^{-1}\rVert^{-1}
	= \log\,\lVert df^{-n}_{/F_{f^n(x)}}\rVert^{-1}
\]
is also subadditive.
By Kingman's subadditive ergodic theorem, there exist measurable functions $\lambda_E,\lambda_F\colon\Lambda\to\bR$ such that for a given measure $\mu\in\cM_f(\Lambda)$ for $\mu$-almost every $x$ we have
\[
	\lambda_E(x)
	 = \lim_{n\to\infty}\frac1n\log\,\lVert df^n_{/E}\rVert,
	 \quad
	\lambda_F(x)
	 = \lim_{n\to\infty}\frac1n\log\,\lVert (df^n_{/F})^{-1}\rVert^{-1}
\]
and are called the \emph{maximal Lyapunov exponent} of $x$ in $E$ and the \emph{minimal Lyapunov exponent} of $x$ in $F$, respectively.
(For this note that in our setting the sequences $(\psi_n/n)_n$ and $(\phi_n/n)_n$ are uniformly bounded from below.) 
Moreover, 
\begin{equation}\label{eq:maxminexp}
\begin{split}
	\lambda_E(\mu)
	& \eqdef 
	\int\lambda_E\,d\mu
	= \lim_{n\to\infty}\int\frac1n\psi_n\,d\mu 
	= \inf_{n\ge1}\int\frac1n\psi_n\,d\mu \\
	\lambda_F(\mu)
	& \eqdef 
	\int\lambda_F\,d\mu
	= \lim_{n\to\infty}\int\frac1n\phi_n\,d\mu 
	= \sup_{n\ge1} \int\frac1n\phi_n\,d\mu
\end{split}
\end{equation}
and we call these numbers the \emph{maximal Lyapunov exponent of $\mu$ in $E$} and the \emph{minimal Lyapunov exponent of $\mu$ in $F$}, respectively. 
If $\mu$ is ergodic, then $\lambda_E$ and $\lambda_F$ are constant $\mu$-almost everywhere.

Given $N\ge1$, define the function $\psi_N^\ast\colon\Lambda\to\bR$ by
\[
	\psi_N^\ast(x)
	\eqdef \sup_{k\ge1}\frac1k\sum_{\ell=0}^{k-1}\psi_N(f^{\ell N}(x)).
\]
Let
\[
	B_N
	\eqdef \left\{x\colon \psi_N^\ast(x)\le -1\right\}.
\]
Recall that, given $\mu\in\cM_f$,
by the maximal ergodic theorem applied to the Birkhoff averages (with respect to $f^N\colon\Lambda\to\Lambda$) of the potential $\psi_N$ we have
\[
	\int_{\Lambda\setminus B_N}\left(\psi_N+1\right)\,d\mu
	\ge0
\]
which implies
\begin{equation}\label{eq:maxergthe}
	- \int_{\Lambda\setminus B_N}\psi_N\,d\mu
	\le  \mu(\Lambda\setminus B_N).
\end{equation}

\section{$C^1$-dominated Pesin theory}
\label{sec:domPis}

The usual Pesin theory requires a $C^{1+\alpha}$ regularity of the dynamics (\cite{Pug:84,BonCroShi:13}  
show that the Pesin theory does not hold in the $C^1$-setting).   
We use here the $C^1$-dominated Pesin theory developed in~\cite{AbdBonCro:11,AviBoc:12}, which  holds 
only when the stable/unstable bundle
is dominated. From this theory, we derive  in particular the size of local stable/unstable manifolds. 

\subsection{Invariant manifolds in the $C^1$  dominated setting}

Let $f$ be a $C^1$ diffeomorphism of a Riemannian manifold $M$. The \emph{(Pesin) stable  set}  at a point $x\in M$ is defined by
\[
	W^s(x)\eqdef 
	\{y\in M\colon \limsup_{n\to\infty}\frac1n\log d(f^n(y),f^n(x))<0\}\,.
\]
The \emph{(Pesin) unstable  set} at $x$ is defined to be the stable  set at 
$x$ with respect to $f^{-1}$ and denoted by $W^u(x)$.

In this section we continue to assume that $\Lambda\subset M$ is a compact $f$-invariant set and that its tangent bundle carries a dominated splitting $T_\Lambda M=E\oplus_{_<} F$. 
We fix cone fields $\cone^E,\cone^F$ around the  bundles $E,F$ which are strictly invariant. 
More precisely, for every $x\in\Lambda$ the open cone $\cone^E_x\subset T_xM$ contains $E_x$, is 
transverse to $F_x$, and the image of its closure under $df_x$ is contained in $\cone^E_{f(x)}$. 
Analogously we define $\cone^F_x$ being invariant with respect to $df_x^{-1}$. Such cone fields 
can be extended to cone fields on a small open neighborhood of $\Lambda$, keeping all the above 
given properties. 

Given $x\in \Lambda$ and $\delta>0$, we say that a set $D$ is a $C^1$ \emph{stable disk} of radius 
$\delta$ centered at $x$ if there is a $C^1$ map $\varphi$ from a ball of radius $\delta$ centered 
in $0$ in $E_x$ to $F_x$ such that $D$ is the graph of the map $v\mapsto \exp_x(v+\varphi(v))$ and 
so that the tangent space of $D$ at each point  is tangent to the cone field $\cone^E$. 
Analogously, we define a $C^1$ \emph{unstable disk}.  

\begin{definition}[$\delta_0$]\label{def:plaques}
By the Plaque family theorem (see~\cite[Theorem 5.5]{HirPugShu:77} or~\cite{AbdBonCro:11}) there exist
some $\delta_0$ and a continuous family of $C^1$ stable (unstable) disks  centered at points $x\in\Lambda$ 
of radius $\delta_0$ which is locally  invariant. Moreover, by choosing $\delta_0$ small 
enough, one can ensure that these disks are tangent to the cone field $\cone^E$ ($\cone^F$). 
By shrinking if necessary the size of the plaques one may assume 
that every stable plaque is transverse to any unstable
plaque at any point of intersection.

For the following we will fix such families and denote them by 
$\{D^E_x\}_{x\in\Lambda}$ and $\{D^F_x\}_{x\in\Lambda}$.
\end{definition}

In the whole paper, when we consider a set $\Lambda$ (or $\tilde\Lambda$) 
with a dominated splitting $E\oplus_{_<}F$, we always endow it implicitly with a continuous family of
locally invariant  plaques.  We say that a point $x\in\Lambda$ has a \emph{(local) stable} (\emph{unstable}) \emph{manifold of size
$\delta$} if there is a disk $D$ of radius $\delta$ centered at $x$ in the plaque $D_x^E$ ($D^F_x$) and 
contained in the stable set of $x$.

The following lemma is an immediate consequence of the transversality of the stable and unstable plaques families $\{D^E_x\}_{x\in\Lambda}$ and 
$\{D^F_x\}_{x\in\Lambda}$ and will be often used in this paper.

\begin{lemma}\label{lem:prostru}
For every $\delta>0$ there exists $\eta>0$ such that for every pair of points 
$x,y\in\Lambda$ which satisfy $d(x,y)\le\eta$ and which both have $C^1$  stable and unstable 
disks of radius at least $\delta$ these disks intersect transversally and cyclically, that is, 
the stable disk of $x$ intersects the unstable disk of $y$ and vice versa.
\end{lemma}

\subsection{Size of invariant manifolds}

The sets and functions considered below are the main ingredients in \cite{AbdBonCro:11,AviBoc:12} for controlling the size of the invariant manifolds. Here we favor the more direct and flexible (though essentially equal) approach in~\cite[Sections 4.2.1--4.2.2]{AviBoc:12} since to determine the sizes of stable/unstable manifolds it does not require to verify certain quantifiers on a full measure set of points with respect to some (ergodic) measure (compare~\cite[Proposition 8.9]{AbdBonCro:11}).

Given  a positive integer $N$, define the \emph{stable Pesin block}
\[
	B^s(N,f)
	\eqdef \left\{x\in\Lambda\colon\sup_{k\ge1}\frac1k
		\prod_{\ell=0}^{k-1}\log\,\lVert df^N_{/E_{f^{\ell N}(x)}}\rVert
		\le -1\right\}.
\]
Also define by $B^u(N,f)\eqdef B^s(N,f^{-1})$ the \emph{unstable Pesin block}.

\begin{proposition}[{\cite[Theorem 4.7]{AviBoc:12}}]\label{pro:AB}
	Let $f$ be a $C^1$ diffeomorphism, $\Lambda$ a compact invariant set with a dominated splitting $E\oplus_{_<}F$. For every $N\ge 1$ there exists $\delta>0$ such that for every $x\in  B^s(N,f)$ there exists an injectively immersed $C^1$-manifold $W^E(x)$ of dimension $\dim E$ tangent to $E_x$, which is a local stable manifold $W^E(x)\subset W^s(x)$ of size $\delta$. The analogous result holds for $B^u(N,f)$.
\end{proposition}

\subsection{Approximation of hyperbolic measures}

We derive the following consequence from Proposition~\ref{pro:AB} about hyperbolic measures and their neighborhoods.

\begin{proposition}\label{pro:ABapprox}
Assume that the Standing hypothesis is satisfied.
	Let $\mu\in\cM_f(\Lambda)$ be hyperbolic with $s$-index $\dim E$.
	Then for every $\varepsilon\in(0,1)$ 	there exist  a positive integer $N$ and a number $\varkappa>0$ such that for any $\nu\in\cM_f(\tilde\Lambda)$ satisfying $D(\nu,\mu)<\varkappa$ we have $\nu(B^s(N,f)\cap B^u(N,f))>1-\varepsilon$.
Moreover, any ergodic $\nu\in\cM_f(\tilde\Lambda)$  satisfying $D(\nu,\mu)<\varkappa$ is hyperbolic with $s$-index $\dim E$ and satisfies $\lambda_E(\nu)<\lambda_E(\mu)/4$ and $\lambda_F(\mu)/4<\lambda_F(\nu)$.	 
\end{proposition}

\begin{proof}
The proof follows closely arguments in~\cite[Sections 4.2.1--4.2.2]{AviBoc:12} invoking the subadditive and maximal ergodic theorems which we recalled in Section~\ref{subsec:submaxergthe}. 

Let $\psi_n=\log\,\lVert df^n_{/E}\rVert$ and observe that $(\psi_n)_n$ is a subadditive sequence of continuous functions which are uniformly bounded from below.  By Kingman's subadditive ergodic theorem, $\psi_n/n$ converges $\mu$-almost everywhere to a measurable function $\lambda_E$. Since $\mu$ is hyperbolic with $s$-index $\dim E$, $\lambda_E(x)$ is negative at $\mu$-almost every $x$.
	
Let $\varepsilon\in(0,1)$. There exist $\lambda>0$ and a measurable set $\Omega\subset\Lambda$ such that $\mu(\Omega)>1-\varepsilon$ and such that for every $x\in \Omega$ we have $\lambda_E(x)\le -\lambda$.
	
Let $\gamma\in(0,\lambda)$ and $\chi=\lambda-\gamma$. Note that $\lambda_E+\chi<0$ on $\Omega$.
		
By the dominated convergence theorem, $\lambda_E$ is integrable and 
\[
	\lim_{n\to\infty}\int\Big\lvert \frac{\psi_n}{n}-\lambda_E\Big\rvert\,d\mu
	=0.
\]	
Hence, denoting $a^+=\max\{0,a\}$, we have
\[
	\int \Big\lvert \frac{\psi_n}{n}-\lambda_E\Big\rvert\,d\mu
	\ge \int \Big(\frac{\psi_n}{n}-\lambda_E\Big)^+\,d\mu
	\ge0
\]
which implies 
\begin{equation}\label{eq:conclude}
	\lim_{n\to\infty}\int \Big(\frac{\psi_n}{n}-\lambda_E\Big)^+\,d\mu=0.
\end{equation}
Moreover, since $\lambda_E+\chi<0$ on $\Omega$, we can conclude that for any $x\in\Omega$ we have
\[
	0
	\le \Big(\frac{\psi_n}{n}+\chi\Big)^+(x)
	=  \Big(\Big(\frac{\psi_n}{n}-\lambda_E\Big)+(\lambda_E+\chi)\Big)^+(x)
	\le \Big(\frac{\psi_n}{n}-\lambda_E\Big)^+(x),
\]
while for $\mu$-almost every $x$ in the bigger set $\Lambda$, since $ \lambda_E(x)<0$ we have 
\[
	0
	\le \Big(\frac{\psi_n}{n}+\chi\Big)^+(x)
	=  \Big(\Big(\frac{\psi_n}{n}-\lambda_E\Big)+(\lambda_E+\chi)\Big)^+(x)
	\le \Big(\frac{\psi_n}{n}-\lambda_E\Big)^+(x)+\chi.
\]
Moreover, since $\mu(\Lambda\setminus\Omega)\le\varepsilon$ we obtain
\[\begin{split}
	0\le \int \Big(\frac{\psi_n}{n}+\chi\Big)^+\,d\mu
	&\le \chi\mu(\Lambda\setminus\Omega)
			+ \int\Big(\frac{\psi_n}{n}-\lambda_E\Big)^+\,d\mu\\
	&\le \chi\varepsilon
			+ \int\Big(\frac{\psi_n}{n}-\lambda_E\Big)^+\,d\mu		.
\end{split}\]
 Hence, with~\eqref{eq:conclude} we conclude 
\[
	0
	\le \lim_{n\to\infty}\int \Big(\frac{\psi_n}{n}+\chi\Big)^+\,d\mu
	\le \chi\varepsilon.	
\]

We now choose $N\ge1$ such that 
\begin{equation}\label{eq:choiceconstants}
	\frac{1}{N}<\chi\varepsilon
	\quad\text{ and }\quad
	\int \Big(\frac{\psi_N}{N}+\chi\Big)^+\,d\mu
	<2\chi\varepsilon.
\end{equation}

As the integrand of the integral in~\eqref{eq:choiceconstants} is a continuous function, there exists $\varkappa>0$ such that for every $\nu\in\cM_f(\tilde\Lambda)$ satisfying $D(\nu,\mu)<\varkappa$ we have
\begin{equation}\label{eq:proxmeas}
	\int \Big(\frac{\psi_N}{N}+\chi\Big)^+\,d\nu
	<3\chi\varepsilon.
\end{equation}

By the maximal ergodic theorem~\eqref{eq:maxergthe}  we have
\begin{equation}\label{eq:maxergtheproof}
	-\int_{\tilde\Lambda\setminus B^s(N,f)}\psi_N\,d\nu
	\le \nu\left(\tilde\Lambda\setminus B^s(N,f)\right).
\end{equation}
What remains is to estimate $\nu(B^s(N,f))$ for every such measure $\nu$. Observe that with~\eqref{eq:proxmeas}  
\begin{equation}\label{eq:choiceN}
	\int_{\tilde\Lambda\setminus B^s(N,f)} \Big(\frac{\psi_N}{N}+\chi\Big)^+\,d\nu
	\le \int \Big(\frac{\psi_N}{N}+\chi\Big)^+\,d\nu
	< 3\chi\varepsilon.
\end{equation}
As  $\chi=(t+\chi)-t\le(t+\chi)^+-t$ for any $t\in\bR$, with $t=\psi_N/N$ we have
\[
	\chi\nu\left(\tilde\Lambda\setminus B^s(N,f)\right)
	= \int_{\tilde\Lambda\setminus B^s(N,f)}\chi\,d\nu
	\le \int_{\tilde\Lambda\setminus B^s(N,f)}
		\Big(\Big(\frac{\psi_N}{N}+\chi\Big)^+-\frac{\psi_N}{N}\Big)\,d\nu.
\]
With~\eqref{eq:choiceN},~\eqref{eq:maxergtheproof}, the fact that $\nu$ is a probability measure, and~\eqref{eq:choiceconstants} we obtain
\[
	\chi\nu\left(\tilde\Lambda\setminus B^s(N,f)\right)
	< 3\chi\varepsilon
		+\frac1N\nu\left(\tilde\Lambda\setminus B^s(N,f)\right)	
	<4\chi\varepsilon	.
\]
From this we conclude $\nu(B^s(N,f))>1-4\varepsilon$.

Without loss of generality, we can assume that the above quantifiers and estimates  simultaneously hold also for $B^u(N,f)$. 
This proves the claim about the size of stable and unstable Pesin blocks.

To prove the claim about hyperbolicity of ergodic measures close to $\mu$, consider now also the subadditive sequence $\phi_n=-\log\,\lVert df^{-n}_{/F_{f^{-n}(x)}}\rVert$ and let us assume that above $N$ was chosen large enough such that we have
\[
	\lambda_E(\mu)
	\le\int\frac{\psi_N}{N}\,d\mu
	<\frac12\lambda_E(\mu)
	,\quad
	\frac12\lambda_F(\mu)
	\le\int\frac{\phi_N}{N}\,d\mu
	<\lambda_F(\mu).
\]
By our assumption we have $\lambda_E(\mu)<0<\lambda_F(\mu)$.
As $\phi_N/N$ is continuous, we can assume that $\varkappa$ was chosen small enough such that for every $\nu\in\cM_f(\tilde\Lambda)$ satisfying $D(\nu,\mu)$ we also have 
\[
	\int\frac{\psi_N}{N}\,d\nu<\frac14\lambda_E(\mu)
	<0<
	\frac14\lambda_F(\mu)
	\le\int\frac{\phi_N}{N}\,d\mu.
\]
Then by the subadditive ergodic theorem~\eqref{eq:maxminexp}, for every ergodic measure $\nu$ satisfying $D(\nu,\mu)<\varkappa$ we obtain
\[
	\lambda_E(\nu)
	<\frac14\lambda_E(\mu)
	<0<
	\frac14\lambda_F(\mu)
	<\lambda_F(\nu).
\]
Thus, $\nu$ is hyperbolic with $s$-index $\dim E$ with Lyapunov exponents uniformly bounded away from $0$.
\end{proof}
The following is an immediate consequence of Proposition~\ref{pro:ABapprox}.

\begin{corollary}\label{corpro:allperiodicarelarge}
	Assume that the Standing hypothesis is satisfied. Let $\mu\in\cM_f(\Lambda)$ be    hyperbolic with $s$-index $\dim E$. Then for every $\varepsilon\in(0,1)$ there exist positive numbers  $\delta>0$ and $\varkappa>0$ such that	
for every invariant measure $\nu\in\cM_f(\tilde\Lambda)$ supported on the orbit of a periodic point $p$ and satisfying $D(\nu,\mu)<\varkappa$ there is a set of points $G(p)$ on this orbit  satisfying
\begin{itemize}
\item $\card G(p)\ge (1-\varepsilon)\pi(p)$, where $\pi(p)$ denotes the period of $p$;
\item  every $x\in G(p)$  has a local stable and a local unstable manifold of size $\delta$, respectively.
\end{itemize}
\end{corollary}

\begin{proof}
Taking $N=N(\mu,\varepsilon)$ and $\varkappa=\varkappa(\mu,\varepsilon)$ as in Proposition~\ref{pro:ABapprox} and $\delta=\delta(N)$ as in Proposition~\ref{pro:AB}, it suffices to observe that if $G$ is a Borel set such that $\nu(G)> 1-\varepsilon$ then with $G(p)=G\cap\{p,f(p),\ldots,f^{\pi(p)-1}(p)\}$ we have $\card G(p)>(1-\varepsilon)\pi(p)$.
\end{proof}

We will finally formulate the following slightly strengthened version of~\cite[Proposition 1.4]{Cro:11} (and of Proposition~\ref{pro:ABapprox}) which is contained in its proof in~\cite{Cro:11}.
This is an ersatz to Katok's horseshoe construction (see~\cite[Supplement S.5]{KatHas:95}) in the $C^1$ dominated setting.

\begin{corollary}\label{corpro:sylvnew}
	Assume that the Standing hypothesis is satisfied. Let $\mu\in\cM_{\rm erg}(\Lambda)$ be hyperbolic with $s$-index $\dim E$.

	Then for every $\varepsilon\in(0,1)$ there exist $\delta>0$ and $\varkappa>0$  such that for every $\nu\in\cM_{\rm erg}(\tilde\Lambda)$ satisfying $D(\nu,\mu)<\varkappa$ there exists a set $\Gamma_\nu$ such that $\nu(\Gamma_\nu)>1-\varepsilon$ and that for every point $x\in\Gamma_\nu$ there is a sequence $(p_n)_n\subset\tilde\Lambda$ of hyperbolic periodic points with $s$-index $\dim E$ such that:
\begin{itemize}
\item $p_n$ converges to $x$ as $n\to\infty$;
\item the orbit of $p_n$ converges to the support of $\nu$ in the Hausdorff topology;
\item the invariant measures supported on the orbit of $p_n$ converge to $\nu$ in the weak$\ast$ topology;
\item every $p_n$ has a stable and a unstable local manifold of size $\delta$, respectively. 
\end{itemize}
As a consequence, if $\varkappa$ was chosen sufficiently small, all points $p_n$ whose corresponding measures $\nu_n$ satisfy $D(\nu_n,\mu)<\varkappa$ are pairwise homoclinically related. 
\end{corollary}

\begin{proof}
We follow the proof of~\cite[Proposition 1.4]{Cro:11}, with essentially the only change that we use the stable and unstable Pesin blocks as in~\cite{AviBoc:12}. By Proposition~\ref{pro:ABapprox} applied to $\mu$ and $\varepsilon$,  there are a positive integer $N=N(\mu,\varepsilon)$ and a number $\varkappa=\varkappa(\mu,\varepsilon)>0$   such that for every ergodic measure $\nu\in\cM_f(\tilde\Lambda)$ with $D(\nu,\mu)<\varkappa$ we have $\nu(B^s(N,f)\cap B^u(N,f))>1-\varepsilon$. Let $\delta_0=\delta(N)>0$ as in Proposition~\ref{pro:AB} applied to the set $\tilde\Lambda$. Recall that for every $x\in B^s(N,f)\cap B^u(N,f)$ for every $k\ge0$ we have
\[
	\prod_{\ell=0}^{k-1}\lVert df^N_{/E_{f^{\ell N}(x)}}\rVert \le e^{-k}
	\quad\text{ and }\quad
	\prod_{\ell=0}^{k-1}\lVert df^{-N}_{/F_{f^{-\ell N}(x)}}\rVert \le e^{-k}	.
\]
As $\nu$ is assumed to be ergodic, $\nu$-almost every point is recurrent. Hence,  since $\nu(B^s(N,f)\cap B^u(N,f))> 1-\varepsilon>0$ there is a set $\Gamma_\nu\subset B^s(N,f)\cap B^u(N,f)$ 
satisfying $\nu(\Gamma_\nu)>1-\varepsilon$ formed by points $x$ in the support of $\nu$ that have positive iterates $f^m(x) \in B^s(N,f) \cap B^u(N,f)$ arbitrarily close to $x$ such that
the (noninvariant) atomic measure $\frac1m\sum_{i=1}^{m-1}\delta_{f^i(x)}$ is arbitrarily close to $\nu$. In particular,  for every $k=0,\ldots,m$ one has
\begin{equation}\label{eq:syleq}
	\prod_{\ell=0}^{k-1}\lVert df^N_{/E_{f^{\ell N}(x)}}\rVert \le e^{-k}
	\quad\text{ and }\quad
	\prod_{\ell=0}^{k-1}\lVert df^{-N}_{/F_{f^{m-\ell N}(x)}}\rVert \le e^{-k}	.
\end{equation}
This property and the domination $E\oplus_{<} F$ allow to apply Liao-Gan's shadowing lemma~\cite{Gan:02} which states that the orbit segment $x,f(x),\ldots,f^m(x)$ is $c$-shadowed by a periodic orbit $p,f(p),\ldots,f^m(p)=p$ where $c$ tends to $0$ as $d(x,f^m(x))$ decreases. In particular the atomic measure on this periodic orbit is close to $\nu$. Note that the periodic point $p$ satisfies an estimate similar to~\eqref{eq:syleq} up to some multiplicative constants which only depend on $c$ and $N$ (but not on $\nu$) and which tend to $1$ as $d(x,f^m(x))$ decreases. 
This way, we obtain a sequence $(p_n)_n\subset\tilde\Lambda$ of hyperbolic periodic points with $s$-index $\dim E$ which converge to $x$, whose orbit converge to $\supp\nu$, and whose measures converge to $\nu$ in the weak$\ast$ topology. And by the estimates~\eqref{eq:syleq} and arguments similar as in the proof of Proposition~\ref{pro:AB} the size of the local stable and the local unstable manifolds at $p_n$ are uniformly bounded from below by some positive $\delta\in(0,\delta_0)$ (which depends on $N$ but does not depend on $\nu$). But we refrain from giving all details.
\end{proof}

\section{Intersection classes}\label{sec:intcla}

\begin{definition}[Intersection class]
Let $f$ be a diffeomorphism of a compact man\-i\-fold $M$. Let $\Per_{\rm hyp}$  be the set of hyperbolic periodic  points in $M$. On $\Per_{\rm hyp}$ we consider the relation $\sim_H$ of being \emph{homoclinically related}, that is, 
two hyperbolic periodic  points $p\sim_H q$ if, and only if, the invariant manifolds of 
their orbits meet cyclically and transversely. 
This  defines an equivalence relation on the set $\Per_{\rm hyp}$. The equivalence classes for $\sim_H$ are
called \emph{the intersection classes}.%
\footnote{Note that we do not consider the closure of the set of the transverse intersection points. To avoid confusion with other common usage of terms we avoid to call such an equivalence class a \emph{homoclinic class}, although this would be an appropriate name.}
\end{definition}

\begin{remark} \label{rem:homcla}
For every hyperbolic periodic point, its \emph{homoclinic class} is the closure of its intersection class.

 A given homoclinic class $C$ of a hyperbolic periodic point $p$ may contain several distinct intersection classes. For instance, $C$ may contain points of
$s$-indices different from the  $s$-index of $p$. More surprisingly, $C$ may contain hyperbolic periodic points $q$ with the same $s$-index as $p$
but not homoclinically related with $p$: \cite{DiaHorRioSam:09,DiaGel:12} (see also \cite{LepOliRio:11}) 
provide such examples where the homoclinic class of $q$ is strictly contained in $C$. 
\end{remark}

For the remainder of this section, we assume that the Standing hypothesis is satisfied.

\begin{definition}[Intersection class of a hyperbolic ergodic measure]\label{def:intcla}
For any hyperbolic ergodic  measure $\mu$ with $s$-index $\dim E$ supported in $\Lambda$, Corollary~\ref{corpro:sylvnew} in particular  
asserts that for $\mu$-almost every $x$ there exists a sequence $(p_n)_{n\ge1}$ 
of hyperbolic periodic  points $p_n\in\tilde \Lambda$ with $s$-index $\dim E$ such that  $\lim_{n\to\infty}p_n=x$.
On the other hand, Theorem~\ref{t.homoclinic} asserts that there is a neighborhood of $\mu$ in the set of invariant measures supported in $\tilde \La$ so that
every pair of periodic orbits in $\tilde \La$ whose corresponding invariant measures are contained in this neighborhood are homoclinically related, hence belong to the same class. 

Therefore, Corollary~\ref{corpro:sylvnew} together with Theorem~\ref{t.homoclinic} implies that there is exactly one intersection class containing orbits in $\tilde \La$ and whose
corresponding measures tend to $\mu$.  
We call this class \emph{the intersection class of $\mu$} and we denote it by $H(\mu)$.
\end{definition}

For Definition~\ref{def:intcla} to be coherent it remains now to prove Theorem~\ref{t.homoclinic}. 

\begin{proof}[Proof of Theorem~\ref{t.homoclinic}]
We want to prove  that, for $\varkappa=\varkappa(\mu)$ small enough, all periodic orbits in $\tilde \La$ whose corresponding measures are $\varkappa$-close to $\mu$ are homoclinically related. 
Arguing by contradiction, we assume the existence of a sequence $(\cO(p_n), \cO(q_n))_n$ of pairs of periodic orbits of periodic points $p_n,q_n\in\tilde \La$, 
whose corresponding measures $\nu_n,\xi_n$ tend to $\mu$, and so that the unstable manifold of $\cO(p_n)$ has no transverse intersection point with the stable manifold of $\cO(q_n)$. 

Recall that  for $n$ large enough $\cO(p_n)$ and $\cO(q_n)$ are hyperbolic with the same $s$-index  as $\mu$. Furthermore, Corollary~\ref{corpro:allperiodicarelarge} ensures the existence 
of numbers $\delta=\delta(\mu)>0$ and $\theta=\theta(\mu)\in(0,1)$  and of a subset  $G(p_n)\subset \cO(p_n)$ with the following properties:
\begin{itemize}
 	\item $\card(G(p_n))\ge\theta\card(\cO(p_n))$ 
		and $\card(G(q_n))\ge\theta\card(\cO(q_n))$, that is, 
		$\nu_n(G(p_n))\geq \theta$ and $\xi_n(G(q_n))\geq\theta$.
 	\item for any $x\in G(p_n)$ the local stable (and unstable) manifold of $x$ contains 
		a disk of radius $\delta$ centered at $x$, and for any $y\in G(q_n)$ 
		the local unstable (and stable) manifold contains a disk of radius $\delta$ centered at $y$. 
\end{itemize}

Let $\nu^s_n$ denote the restriction of $\nu_n$ to the set $G(p_n)$ and $\xi^u_n$ denote the restriction of $\xi_n$ to $G(q_n)$. 
Thus $\nu^s_n$ and $\xi^u_n$ are positive measures whose total mass is within the interval $[\theta,1]$.  Note that the set of such measures is  closed with respect to the weak$\ast$ topology. 

Therefore, up to considering  subsequences, one may assume that the measures $\nu^s_n$ converge to a measure $\nu^s$ and the measures $\xi_n^u$ tend to a measure $\xi^u$. 
Note that for any continuous positive function $\phi$ one has $\int \phi \,d\nu_n^s\leq \int\phi \,d\nu_n$ which implies that $\int \phi \,d\nu^s\leq\int \phi \,d\mu$.  In the same way, one deduces $\int\phi \,d\xi^u \leq \int \phi \,d\mu$. 
Thus, 
\[
	\mu(\supp\nu^s)\ge\theta 
	\quad\text{ and }\quad 
	\mu(\supp\xi^u)\ge\theta.
\]	
As $\mu$ is ergodic, there is $k\in \NN$ so that 
\[
	\mu\big(\supp\nu^s\cap  f^k(\supp\xi^u)\big)>0.
\]	 
Choose now positive numbers $\delta_0$ and $\eta$ satisfying $\delta_0<\delta\cdot(\max\|Df^{-1}\|)^{-k}$ and $\eta\ll\delta_0$. 
Consider a point $x\in \supp(\nu^s)\cap  f^k(\supp\xi^u)$ and the ball $B(x,\eta)$ of radius $\eta$ centered at $x$.  In particular $\nu^s(B(x,\eta))>0$ and $\xi^u(f^{-k}(B(x,\eta)))>0$.

For any $n$ large enough we have $\nu^s_n(B(x,\eta))>0$ and  $\xi^u_n(f^{-k}(B(x,\eta)))>0$. This means, in particular, that the set $B(x,\eta)$ contains:
\begin{itemize}
	\item  a point $y_n\in G(p_n)\subset \cO(p_n)$ 
		(with local stable manifold larger than $\delta$),
	\item  a point $f^k(z_n)$ with $z_n\in G(q_n)\subset \cO(q_n)$ 
		($z_n$ has a local unstable manifold larger than $\delta$ and therefore $f^k(z_n)$ has a unstable manifold larger than $\delta_0$).
\end{itemize}

When $\eta$ was chosen small enough, this implies that $W^s(y_n)$ cuts transversely $W^u(f^k(z_n))$. But this 
contradicts the hypothesis that $W^s(\cO(p_n))$ does not have transverse intersection points
with $W^u(\cO(q_n))$,  proving the theorem.
\end{proof}

A straightforward consequence is the following.

\begin{corollary}
	For $\mu$ satisfying the hypotheses of Corollary~\ref{corpro:sylvnew}, the closure $\overline{H(\mu)}$ of the 
	intersection class $H(\mu)$ contains the support of $\mu$.
\end{corollary}

\begin{remark}\label{rem:hertzens}
Finally let us point out that there exist other concepts in the literature which are related to the above defined intersection class (besides the already mentioned concept of a homoclinic class, see Remark~\ref{rem:homcla}). Generally they are, however, quite distinct objects. 

For example, a related concept has been studied in~\cite{RodRodTahUre:11} (see also~\cite{AviBoc:12}). For completeness we will provide its definition. 
In the context of $f$ being a $C^{1+\alpha}$ 
diffeomorphism, Pesin theory guarantees that for every ergodic hyperbolic  measure $\mu$ there is a Borel set $\widehat\cRR$ of full measure such that for every $x\in\widehat\cRR$ the stable and unstable sets $W^s(x)$ and $W^u(x)$ are immersed manifolds (see~\cite{BarPes:01} for details).
Let $p$ be a hyperbolic periodic   point and denote by $\cO(p)$ its orbit. Define
\[\begin{split}
	H^s_{\rm erg}(p)&\eqdef
	\big\{x\in\widehat\cRR\colon W^s(x)\text{ transversally intersects }
	 W^u(\cO(p))\big\},\\
	H^u_{\rm erg}(p)&\eqdef
	\big\{x\in\widehat\cRR\colon W^u(x)\text{ transversally intersects }
	W^s(\cO(p))\big\}\,,\\
	H_{\rm erg}(p)&=H^s_{\rm erg}(p)\cap H^u_{\rm erg}(p)\,.
\end{split}\]
In~\cite{RodRodTahUre:11}, the set $H_{\rm erg}(p)$ is called the \emph{ergodic homoclinic class} of $p$. It has been introduced to study ergodicity of smooth  measures  and it is shown in~\cite{RodRodTahUre:11} that if $f$ preserves a smooth measure $m$ and if both sets $ H^s_{\rm erg}(p)$ and $H^u_{\rm erg}(p)$ both have positive measure, then $H_{\rm erg}(p)$ coincides modulo zero sets with $H^s_{\rm erg}(p)$ and $H^u_{\rm erg}(p)$, and form an ergodic component of $m$.

In our notation, given an ergodic invariant hyperbolic probability measure $\mu$ supported in $\Lambda$ and with $s$-index $\dim $, for every periodic point $p\in\tilde\Lambda$ with invariant measure sufficiently close to $\mu$, we have
\[
	H(\mu) = H_{\rm erg}(p)\cap\Per_{\rm hyp}
\]
and this set is contained in the homoclinic class of $p$.
\end{remark}

\section{Proof of Theorem~\ref{the:sinterclass}}\label{sec:prooof}

Let $f$ be a $C^1$ diffeomorphism, $\Lambda$ be a compact invariant set with a dominated splitting $TM|_\Lambda=E\oplus_{_<}F$. We fix a maximal invariant set $\tilde \Lambda$ in a neighborhood of $\Lambda$ so that the 
 dominated splitting $E\oplus_{_<}F$ extends to $\tilde \Lambda$.
Let  $\mu\in\cM_f(\Lambda)$ is a hyperbolic invariant Borel probability measure with $s$-index $\dim E$. Let $\mu=\int\nu\,d\blambda(\nu)$ be its ergodic decomposition.

\subsubsection*{Assuming that the intersection classes coincide}

Suppose that there exists some intersection class $H\subset\tilde\Lambda$ such that $\blambda$-almost every ergodic measure $\nu$ has intersection class $H(\nu)=H$.
Given $\varepsilon\in(0,1)$, let $\varkappa=\varkappa(\mu,\varepsilon)>0$ as in Proposition~\ref{pro:ABapprox}. 
By Lemma~\ref{lem:finiteapprox}, there exist $\nu_1,\ldots,\nu_k\in\cM_{\rm erg}(\Lambda)$ and positive numbers $\lambda_1,\ldots,\lambda_k$ satisfying $\sum_{j=1}^k\lambda_j=1$ such that $D(\sum_{j=1}^k\lambda_j\nu_j,\mu)<\varepsilon$.
Applying Corollary~\ref{corpro:sylvnew} to each of these measures $\nu_j$, there is a sequence of hyperbolic periodic  points $(p_{j,n})_n$ such that the invariant measures $\nu_{j,n}$ supported on the orbit of $p_{j,n}$ converge to $\nu_j$ in the weak$\ast$ topology. By local maximality, $p_{j,n}\in\tilde\Lambda$.
Moreover, by Theorem~\ref{t.homoclinic} eventually all points $p_{j,n}$ belong to $H(\nu_j)$. 
As by hypothesis the intersection classes $H(\nu_j)$ all coincide, all periodic 
points $p_{j,n}$ are pairwise homoclinically related. Hence, for every $n\ge1$ there is a basic set $\Delta_n\subset\tilde\Lambda$ which contains 
all $p_{j,n}$, $j=1,\ldots,k$. Thus, there is a sequence  $(p_{n,m})_m$  of periodic points $p_{n,m}$ in $\Delta_n$ whose associated invariant measures 
$\eta_{n,m}$ accumulate at the measure $\lambda_1\nu_{1,n}+\cdots+\lambda_k\nu_{k,n}$. By diagonalization, there is a sequence $(\eta_{n,m_n})_n$ tending 
to $\lambda_1\nu_1+\cdots+\lambda_k\nu_k$. As $\varepsilon$ was arbitrary, it follows that $\mu$ is accumulated by ergodic measures in $\cM_f(\tilde\Lambda)$ as claimed.

\subsubsection*{Assuming existence of one accumulated convex combination}

Suppose now that $\mu$ is accumulated by a sequence $(\nu_n)_n\subset\cM_{\rm erg}(\tilde\Lambda)$ of ergodic measures. By Corollary~\ref{corpro:sylvnew}, each $\nu_n$ is accumulated by measures supported on hyperbolic periodic orbits with $s$-index $\dim E$ which are contained in $\tilde\Lambda$. Hence, without loss of generality, we can assume that  each measure $\nu_n$ is supported  on the orbit of a hyperbolic periodic point $q_n$ with $s$-index $\dim E$  in $\tilde\Lambda$.
 By Corollary~\ref{c.homoclinic}, there exists an intersection class $H\subset\tilde\Lambda$ such that eventually, for $n$ large, $H(\nu_n)=H$.

Let  $\nu\in\cM_f(\Lambda)$ be an  ergodic measure which is a typical point in the decomposition of $\mu$, that is, that $\blambda(\cU)>0$ for any neighborhood $\cU\subset\cM_f(\Lambda)$ of $\nu$. Clearly, $\nu$ is hyperbolic with $s$-index $\dim E$. We want to show that $H(\nu) = H$.
Fixing some $\varepsilon\in(0,1)$ let $\delta=\delta(\nu,\varepsilon)>0$ and $\varkappa=\varkappa(\nu,\varepsilon)>0$ as in Corollary~\ref{corpro:sylvnew}.
Consider a neighborhood $\cU\subset\cM_f(\Lambda)$ such that every $\nu'\in \cU$ satisfies $D(\nu',\nu)<\varkappa$. Let
\[
	\nu_\varkappa
	\eqdef \int_{\cU}\nu'\,d\blambda(\nu')
	\quad\text{ and }\quad
	\lambda_0
	\eqdef \blambda\left(\cU\right)
	\in (0,1].
\]
Note that $\nu_\varkappa$ is finite and inner regular%
\footnote{Recall that a finite measure $\nu$ on a metric space $X$ is \emph{inner regular} (or \emph{tight}) if for any $\varepsilon>0$ there exists a compact subset $\Omega_\varepsilon$ of $\Omega$ such that $\nu(X\setminus\Omega_\varepsilon)<\varepsilon$.}.

On the one hand, by Corollary~\ref{corpro:sylvnew}  for every ergodic measure $\nu'\in \cU$ there is a set $\Gamma_{\nu'}$ so that $\nu'(\Gamma_{\nu'})>1-\varepsilon$ and that every $x\in\Gamma_{\nu'}$ is accumulated by hyperbolic periodic points in $\tilde\Lambda$ of $s$-index $\dim E$ which have local stable and unstable manifolds of size at least $\delta$. As $\delta$ does not depend on $\nu'$, we obtain
\begin{equation}\label{Gdelta}
	\nu_\varkappa(G(\delta))
	>\varepsilon_0,\quad\text{ where }\quad
	\varepsilon_0\eqdef\lambda_0(1-\varepsilon),
\end{equation}
where $G(\delta)$ denotes the set of points $x\in\Lambda$ which are accumulated by a sequence $(p_n)_n$ of hyperbolic periodic points in $\tilde\Lambda$ of $s$-index $\dim E$ such that each $p_n$ has a local stable and a local unstable manifold of size $\delta$, respectively. 

On the other hand, let $N'=N(\mu,\varepsilon_0/5)$ and $\varkappa'=\varkappa(\mu,\varepsilon_0/5)$ be as in Proposition~\ref{pro:ABapprox} and let $\delta'=\delta(N')$ be as in Proposition~\ref{pro:AB}. 

Let $\eta=\min\{\eta(\delta),\eta(\delta')\}>0$ for numbers $\eta(\delta)$ and $\eta(\delta')$ as in Lemma~\ref{lem:prostru}. Hence for any pair of periodic points which are $\eta$-close and which both have $C^1$ stable and unstable disks of radius at least $\min\{\delta,\delta'\}$ these disks intersect transversally and cyclically. 

Let us sketch the final step of the proof. We want to argue that any periodic orbit measure $\nu_n$ sufficiently close to $\mu$ has a large fraction of points on its support which simultaneously have large stable/unstable manifolds and are $\eta/2$-close to the set $G(\delta)$ and hence are $\eta$-close to periodic points accumulating on points in $G(\delta)$ and which have large stable/unstable manifolds and hence belong to the intersection class $H(\nu)$. As their manifolds must intersect transversally and cyclically,  we can conclude that $H(\nu)=H(\nu_n)=H$.
The technical problem is that $G(\delta)$ is in general only a measurable set. To deal with weak$\ast$ approximation, we use inner regularity of $\nu_\varkappa$ and choose some compact subset $\Gamma$ of $G(\delta)$ whose characteristic function can be approximated by some continuous function $\varphi$. This will make our approximation arguments rigorous.

Recalling preliminary facts about the weak$\ast$ topology in Section~\ref{ssec:choquet}, without loss of generality, we can assume that
\[
	\cU
	= \left\{\nu'\in\cM_f(\Lambda)\colon
		\left\lvert\int\varphi_i\,d\nu'-\int\varphi_i\,d\nu\right\rvert<\varkappa_0,
		i=1,\ldots,n\right\}
\]
for some $\varkappa_0>0$ and some continuous functions $\varphi_i\colon M\to\bR$, $i=1,\ldots,n$. Let 
\[
	\Omega
	\eqdef \left\{ x\in G(\delta)\colon 
		\left\lvert \lim_{k\to\infty}\frac1k\sum_{\ell=0}^{k-1}\varphi_i(f^\ell(x)) 
			- \int\varphi_i\,d\nu\right\rvert<\varkappa_0,i=1,\ldots,n
	\right\}
\]
be some subset of the set of generic points for ergodic measures in $\cU$. This set is measurable and by~\eqref{Gdelta} we have $\nu_\varkappa(\Omega)>\varepsilon_0$. Hence, as $\nu_\varkappa$ is regular, $\Omega$ contains a compact set $\Gamma$ such that $\nu_\varkappa(\Gamma)>4\varepsilon_0/5$. Denote by $U$ the $\eta/2$-neighborhood of $\Gamma$.
Let $\varphi\colon M\to[0,1]$ be a continuous function being equal to $0$ outside $U$  and equal to $1$ inside the $\eta/4$-neighborhood of  $\Gamma$.

Consider now $\nu_n$ with $n$ sufficiently large such that 
\[
	\left\lvert \int\varphi\,d\nu_n-\int\varphi\,d\mu\right\rvert
	<\frac{\varepsilon_0}{5}.
\]
By this and the above we have
\[
	\nu_n(U)
	\ge \int\varphi\,d\nu_n
	> \int\varphi\,d\mu -\frac{\varepsilon_0}{5}.
\]
Recalling the ergodic decomposition of $\mu$ we have
\[\begin{split}
	\int\varphi\,d\mu 
	= \int\left(\int\varphi(x)\,d\nu'(x)\right)d\blambda(\nu') 
	\ge \int_{\cU}\left(\int_\Gamma 1\,d\nu'\right)d\blambda(\nu')
	= \nu_\varkappa(\Gamma)
	> \frac{4\varepsilon_0}{5}.
\end{split}\]
Hence, on one hand we obtain $\nu_n(U)> 3\varepsilon_0/5$ and so at least a $3\varepsilon_0/5$-fraction of points on the orbit of $q_n$ are $\eta/2$-close to a point in  $\Gamma\subset G(\delta)$ which by the above choices is accumulated by hyperbolic periodic points with manifolds of size $\delta$. On the other hand we had that at least a $(1-\varepsilon_0/5)$-fraction of points on its orbit have local manifolds of size at least $\delta'$. Hence, at least a $2\varepsilon_0/5$-fraction of points on this orbit have both properties.

This finishes the proof of the theorem.
\qed

\section{Examples}

Let us present some examples to which we apply our results. For that recall that the \emph{convex hull} $\ch(A)$ of a subset $A$ of a vector space is the smallest convex set containing $A$.  

In this section we will always assume that the Standing hypothesis is satisfied.

\begin{example}[Convex hull of finitely many hyperbolic ergodic measures]
Let $\mu_1$, $\ldots$, $\mu_k\in\cM_f(\Lambda)$ be $k$ ergodic hyperbolic measures with $s$-index  equal to $\dim E$. Then the convex hull of all these measures is the simplex 
\[
	\ch(\{\mu_1,\ldots,\mu_k\})
	= \Big\{s_1\mu_1+\cdots+s_k\mu_k\colon 
			s_1,\ldots,s_k\ge0,\sum_{j=1}^ks_j=1\Big\}. 
\]
It is an immediate consequence of Proposition~\ref{pro:ABapprox} that for every $\varepsilon\in(0,1)$ there are positive numbers $\varkappa$,  $\lambda$, and $\delta$ with the following property:
Every ergodic $\nu\in \cM_f(\tilde \Lambda)$ which satisfies $D(\nu,\mu)<\varkappa$ for some $\mu\in\ch(\{\mu_1,\ldots,\mu_k\})$ is hyperbolic with $s$-index $\dim E$ and  the maximal Lyapunov exponent of $\nu$ in $E$ and the minimal Lyapunov exponent  of $\nu$ in $F$ are bounded away from $0$ by $-\lambda$ and $\lambda$, respectively. Moreover, together with Proposition~\ref{pro:AB}, for every such measure $\nu$ there exists a set of points $\Omega_\nu\subset\tilde\Lambda$ satisfying $\nu(\Omega_\nu)>1-\varepsilon$ such that every $x\in\Omega_\nu$ has stable and unstable local manifolds of size at least $\delta$. Finally, by Theorem~\ref{the:sinterclass}, any $\mu\in\ch(\{\mu_1,\ldots,\mu_k\})$ is accumulated by ergodic measures $(\nu_n)_n\subset\cM_f(\tilde\Lambda)$ if, and only if, $H:=H(\mu_1)=\ldots=H(\mu_k)$; and in this case for every $\nu_n$ satisfying $D(\nu_n,\mu)<\varkappa$ we have $H(\nu_n)=H:=H(\mu)$.
\end{example}

\begin{example}[Convex hull of measures supported on hyperbolic sets]
Let $\Lambda_1$, $\ldots$, $\Lambda_k\subset \Lambda$ be $k$ pairwise disjoint uniformly hyperbolic transitive%
\footnote{As in many recent works, an invariant compact set is called \emph{transitive} if it is the closure of a positive orbit. 
This notion is equivalent to the notion of \emph{topological ergodicity}.}  compact invariant sets with $s$-index $\dim E$ and denote $V_j=\cM_f(\Lambda_j)$, $j=1,\ldots,k$.
Then, for every $\varepsilon\in(0,1)$ there are positive numbers $\varkappa$, $\lambda$, and $\delta$ with the following property:
Every ergodic measure $\nu\in \cM_f(\tilde \Lambda)$ which belongs to the $\varkappa$-neighborhood of the convex hull of the (compact) set $V_1\cup\ldots\cup V_k$ is hyperbolic with $s$-index $\dim E$ and for $\nu$ the maximal Lyapunov exponent in $E$ and the minimal Lyapunov exponent in $F$ are bounded away from $0$ by $-\lambda$ and $\lambda$, respectively. Moreover, for every such measure $\nu$ there exists a set of points $\Omega_\nu\subset\tilde\Lambda$ satisfying $\nu(\Omega_\nu)>1-\varepsilon$ such that every $x\in\Omega_\nu$ has stable and unstable local manifolds of size at least $\delta$. Finally, $\Lambda_1,\dots,\Lambda_k$ are pairwise homoclinically related if, and only if,  there exists a measure $\mu\in\ch(V_1\cup\cdots\cup V_k)\setminus (V_1\cup\cdots\cup V_k)$ which is accumulated by ergodic measures $\nu_n$ in $\cM(\tilde \Lambda)$.  
Exactly as for Theorem~\ref{the:sinterclass} we can show that if one measure of the interior\footnote{We leave it as an exercise to see that this set could be empty if the sets $\La_i$ are not pairwise disjoint.} of the set $\ch(V_1\cup\cdots\cup V_k)\setminus (V_1\cup\cdots\cup V_k)$ is accumulated by ergodic measures in $\cM(\tilde \Lambda)$ then the whole convex hull $\ch(V_1\cup\cdots\cup V_k)$ is contained in the closure of periodic orbits homoclinically related with  $\La_1$. 
\end{example}

\begin{example}
	Let $\mu$ be a hyperbolic invariant Borel probability measure with $s$-index $\dim E$ which is supported in $\Lambda$. Assume that $\mu$ is not ergodic and let $\mu=\int\nu\,d\blambda(\nu)$ be its $f$-ergodic decomposition. 
	By hyperbolicity of $\mu$, $\blambda$-almost every $\nu$ is hyperbolic with the same $s$-index  $\dim E$.
	By Theorem~\ref{the:sinterclass}, $\mu$ is accumulated by ergodic measures in $\cM_f(\tilde\Lambda)$ if, and only if, $\blambda$-almost all ergodic measures $\nu$ have the same intersection class $H:=H(\nu)$. Let us assume that this is the case and consider any measure $\blambda'\in\cM(\cM(\Lambda))$ which is absolutely continuous to $\blambda$. Note that then $\blambda'(\cM_{\rm erg}(\Lambda))=1$ and hence $\mu'\eqdef\int\nu\,d\blambda'(\nu)$ is invariant and has $f$-ergodic decomposition $\blambda'$. Moreover, $\blambda'$-almost every $\nu\in\cM_{\rm erg}(\Lambda)$ is hyperbolic with $s$-index $\dim E$ and satisfies $H(\nu)=H$. Thus, we can conclude that   $\mu'$ is also accumulated by ergodic measures in $\cM_f(\tilde\Lambda)$.
\end{example}

\section{Approximation of convex sum of (uniformly) hyperbolic measures, without assuming domination}\label{sec:examples}

In our main results, we require systematically that the measures $\nu_n$ are supported on some dominated extension  $\tilde \La$ of $\La$. 
This domination is necessary for the use of $C^1$-Pesin theory, and a natural question is if we could remove this hypothesis if we assume 
a higher regularity, $C^{1+\alpha}$ or $C^2$. 
The aim of this section is to provide smooth examples showing that the domination is a necessary hypothesis even with high regularity, 
$C^\infty$ or even analytic.

We provide here two classes of examples.

\subsection{Variations of Bowen's figure-8}

We start with a first very classical example.
We consider an area preserving map $f$ of $S^1\times \bR$ given by the time-1 map of the hamiltonian vector field given by the Hamiltonian $H(x,y)=y^2-\cos(4\pi x)$, where $S^1=\bR/\bZ$. Observe that $f$ has 4 fixed points at $(i/4,0)$, $i\in\{1,\ldots,4\}$: there are two   fixed points of center type at $(0,0)$ and $(\frac12,0)$ and two hyperbolic fixed points of saddle type at $p_1=(\frac14,0)$ and $p_2=(\frac34,0)$ having the same contraction/expansion eigenvalues (compare the left figure in Figure~\ref{fig.2}). Observe the following level set 
\[
	\{(x,y)\colon H(x,y)=1\}
	= W^s(p_1)\cup W^s(p_2)=W^u(p_1)\cup W^u(p_2).
\]	 
For every $\varepsilon\in(0,1)$, the level set $\{H=1-\varepsilon\}$ is a smooth closed simple curve on which $f$ is conjugate to a (rational or irrational) rotation (with rotation number tending to $0$ as $\varepsilon\to0$). The following is a classical exercise.

\begin{lemma}\label{lem:Bowen1}
	If $\mu_n$ is a sequence of $f$-invariant probability measures supported on level sets $\{H=1-\varepsilon_n\}$ with $\varepsilon_n\to0$ as $n\to\infty$, then $\mu_n$ converges in the weak$\ast$ topology to $\frac12\delta_{p_1}+\frac12\delta_{p_2}$.
	The Lyapunov exponents of every ergodic measure $\mu_n$ are all equal to  $0$.
	
	The homoclinic and the intersection classes of the points $p_1$, $p_2$ are trivial and hence disjoint. 
\end{lemma} 

Note that this example is not contradicting the results in this paper: the measures $\mu_n$ are not supported in any dominated extension of the hyperbolic set $\Lambda=\{p_1,p_2\}$.

Let us present now two variations of the above example (compare the middle and the right figure in Figure~\ref{fig.2}). 
\begin{itemize}
\item [1)] We can easily modify the map $f$ such that the homoclinic class of $p_1$ is nontrivial while the one of $p_2$ remains trivial (middle figure), keeping all the properties claimed in Lemma~\ref{lem:Bowen1}.   
\item [2)] We can modify the map inside the disk bounded by the heteroclinic connections in such a way that there appears a sequence of hyperbolic periodic saddles whose measures   tend in the weak$\ast$ topology to $\frac12\delta_{p_1}+\frac12\delta_{p_2}$ (right figure). In the case that $f$ is $C^2$ the Lyapunov exponents of these saddles tend to $0$.
\end{itemize}

The construction presented below provides an example where the  hyperbolic measures 
approximating a convex combination of two ergodic hyperbolic measures have Lyapunov exponents all uniformly bounded from $0$.

\begin{figure}
\begin{minipage}[h]{\linewidth}
\centering
\begin{overpic}[scale=.22]{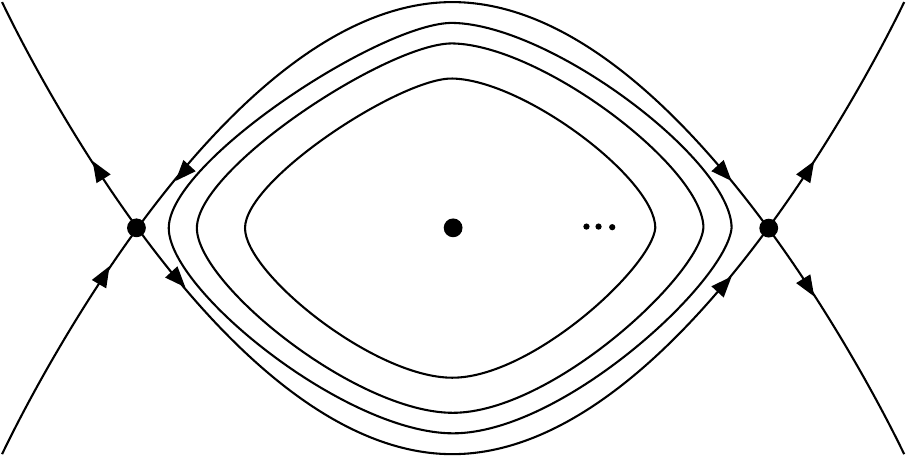}
 \end{overpic}
 \hspace{0.5cm}
\begin{overpic}[scale=.22]{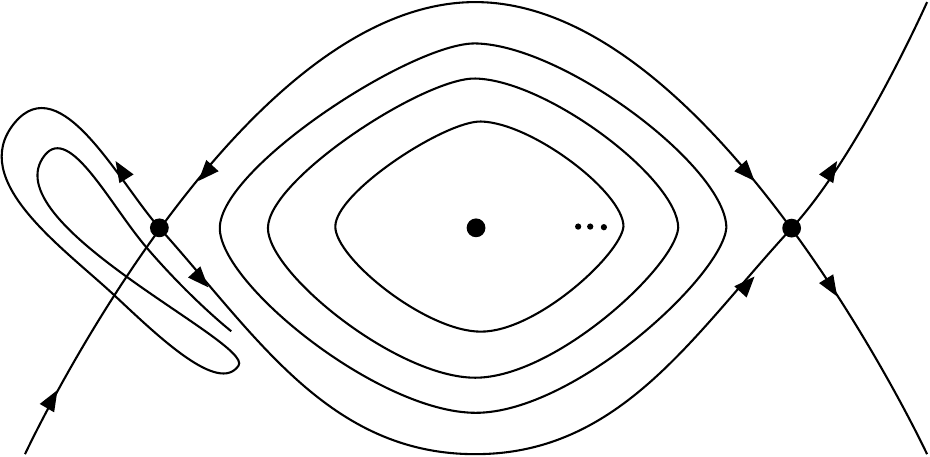}
\end{overpic}
 \hspace{0.5cm}
\begin{overpic}[scale=.22]{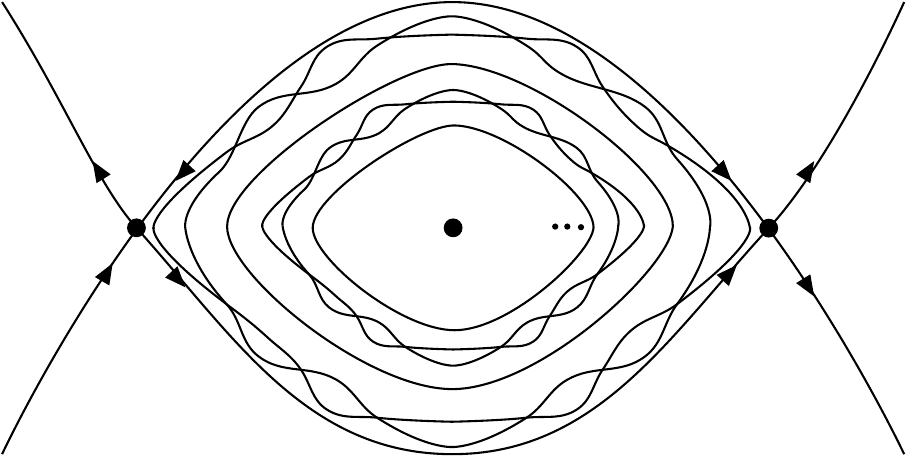}
\end{overpic}
\caption{Variations of Bowen's figure-8}
\label{fig.2}
\end{minipage}
\end{figure}

\subsection{Blowing up of an Anosov diffeomorphism}

Let $A\in SL(2,\ZZ)$ be a hyperbolic linear automorphism with eigenvalues $0<\frac 1\lambda<1<\lambda$. 
We continue denoting by $A\colon \TT^2\to \TT^2$  the induced  linear Anosov diffeomorphism on the torus
$\TT^2$. 

The point $p=(0,0)\in\TT^2$ is a fixed point of $A$. 

Let us denote by $S$ the (non-orientable) closed surface obtained from $\TT^2$ by blowing up the point $p$, and let 
$\pi\colon S\to \TT^2$ be the canonical projection. $S$ is naturally endowed with an $\RR$-analytic structure so that 
the projection $\pi$ is $\RR$-analytic. The projection is a diffeomorphism over $\TT^2\setminus\{(0,0)\}$, and the exceptional fiber 
$C=\pi^{-1}(0,0)$ is canonically identified with the circle $\RR\PP^1$ (see Figure~\ref{fig.1}). 

\begin{figure}
\begin{minipage}[h]{\linewidth}
\centering
\begin{overpic}[scale=.42]{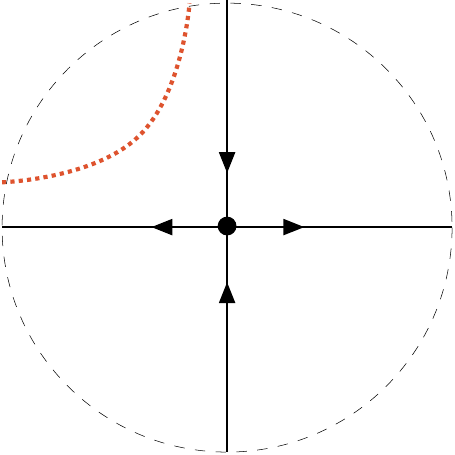}
        \put(54,40){\small $p=(0,0)$}
 \end{overpic}
 \hspace{1.5cm}
\begin{overpic}[scale=.32]{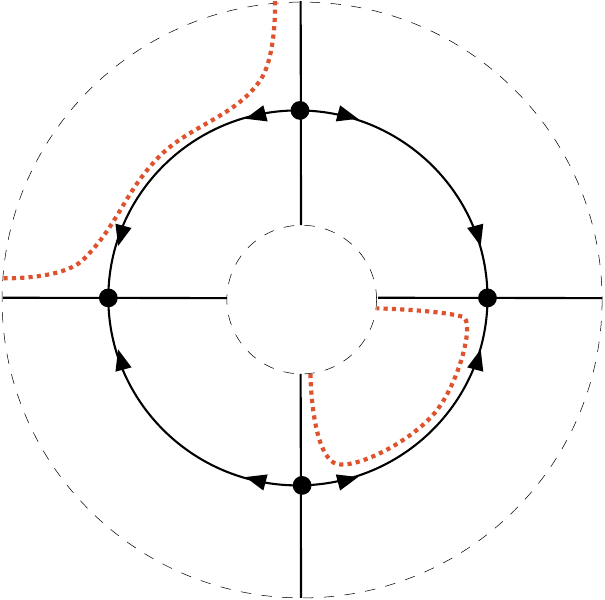}
        \put(23,43){\small $p_1$}
        \put(84,43){\small $p_1$}
        \put(54,85){\small $p_2$}
        \put(54,12){\small $p_2$}
\end{overpic}
\caption{Blowing up the hyperbolic fixed point $p$}
\label{fig.1}
\end{minipage}
\end{figure}
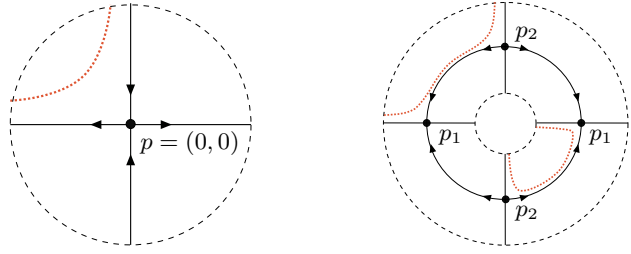

A classical result asserts that:

\begin{lemma}
\begin{itemize}
\item The diffeomorphism $f$ admits a (unique) continuous lift $f_A$ on $S$, and  $f_A\colon S\to S$ is an analytic 
diffeomorphism of $S$. In particular $\pi$ induces a $\bR$-analytic conjugacy between the restrictions of $f_A$ to $S\setminus C$ and of $A$ to $\TT^2\setminus\{p\}$.

\item The restriction of $f_A$  to the exceptional fiber has exactly two fixed points $p_1$ and $p_2$. 
\item The points $p_1$ and $p_2$ are hyperbolic saddle points of $f_A$, whose eigenvalues are $0<\frac 1{\lambda^2}<1<\lambda$ and 
$0<\frac 1{\lambda}<1<\lambda^2$, respectively.
\item the stable manifold $W^s(p_1,f_A)$ is $C\setminus \{p_2\}$ and the unstable manifold $W^u(p_2,f_A)$ is $C\setminus\{p_1\}$. 
As a consequence, the homoclinic classes and intersection classes of $p_1$ and $p_2$ are trivial.
\item $W^u(p_1,f_A)\setminus\{p_1\}$ is the lift of $W^u(p,A)\setminus \{p\}$ and 
$W^s(p_2,f_A)\setminus\{p_2\}$ is the lift of $W^s(p,A)\setminus \{p\}$.
\end{itemize}
\end{lemma}

\begin{theorem}\label{the:exeunihyp}
 Let $\cM_{f_A}$  be the (convex) set of invariant probability measures of $f_A$. Then a measure $\mu\in\cM_{f_A}$ is
 accumulated in  the weak$\ast$ topology by ergodic measures $\nu_n$ if, and only if, 
 \begin{itemize}
  \item either $\mu\in\{\delta_{p_1},\delta_{p_2}\}$, and hence $\nu_n=\mu$ for $n$ large enough,
  \item or $\mu(p_1)=\mu(p_2)$, in which case $\mu$ is the limit of a sequence of periodic orbits whose Lyapunov exponents are 
  $\pm\log\lambda$. 
 \end{itemize}
\end{theorem}

\begin{proof}
	We use the $\bR$-analytic conjugacy between the restriction of $A$ to $\mathbb T^2\setminus \{p\}$ and of $f_A$ to $S\setminus C$. It suffices to observe  that if an orbit of $A$ approaches $p$ then its corresponding orbit of $f_A$ passes the same time close to $C$  with approximately half of this time close to $p_1$ and $p_2$, respectively.	
\end{proof}

As a straightforward corollary one gets the following.

\begin{corollary} The hyperbolic periodic points $p_1$ and $p_2$ have trivial and disjoint homoclinic classes and intersection classes.

Nevertheless the measure $\frac 12 \delta_{p_1} +\frac 12 \delta_{p_2}$ is the weak$\ast$ limit of the measures associated to periodic orbits 
$\cO_n$  of $f_A$ in $S\setminus C$. In particular the orbits $\cO_n$ are hyperbolic and their Lyapunov exponents are $\pm\log \lambda$. 
\end{corollary}

\begin{remark}
	Starting with an Anosov diffeomorphism so that the fixed point $p$ has eigenvalues $\lambda_1<1<\lambda_2$ and blowing up $p$, we get the same result however with measures approximating the convex combination with weights $s$ and $1-s$ satisfying $s/(1-s)=-\log\lambda_1/\log\lambda_2$.
\end{remark}

\begin{remark}
	Multiplying the above dynamics by a Anosov diffeomorphism on ${\mathbb T}^2$ we get an example where both hyperbolic fixed points $p_1$ and $p_2$ have nontrivial but disjoint  homoclinic and intersection classes, keeping all  the properties claimed above.
\end{remark}

\bibliographystyle{amsplain}

\end{document}